\documentclass{amsart} 

\usepackage[latin1]{inputenc}
\usepackage{subfigure}
\usepackage{graphicx}
\usepackage{amssymb}
\usepackage{amsfonts}
\usepackage{amsmath,amscd}
\usepackage{array}
\usepackage{rotating}
\usepackage{epsfig}
\usepackage{xspace}
\usepackage{tikz}

\newtheorem{example}{Example}[section]

\newtheorem{theorem}[example]{Theorem}

\newtheorem{corollary}[example]{Corollary}

\newtheorem{proposition}[example]{Proposition}

\def\<{\langle}
\def\>{\rangle}
\def\NN{{\mathbb N}}    

\def\SG{{\mathfrak S}}  

\long\def\psboxit#1#2{%
\begingroup\setbox0=\hbox{#2}%
\dimen0=\ht0 \advance\dimen0 by \dp0%
    \hbox{%
    \copy0%
    }
\endgroup%
}
\def\SetTableau#1#2#3#4{%
  \gdef\Tabvrule{\vrule\vrule width-0.4pt}
  \gdef\Tabhrule{\hrule\hrule height-0.4pt}  
  \gdef\Tabstrut{\vrule height#1 depth#2 width0pt\relax}
  \gdef\Tabbox##1{\hbox to #3{\hskip0.4pt\hfill\Tabstrut$#4##1$\hfill}}
} 
\def\Case#1{\vcenter{\Tabhrule%
                   \hbox{\Tabvrule\Tabbox{#1}\Tabvrule}\Tabhrule}}

\def\GenTab#1{\vcenter{\halign{&$\Case{##}$\cr#1}}\egroup}
\def\Tableau{%
  \bgroup%
  \let\ =\omit%
  \let\\=\cr%
  \offinterlineskip\GenTab}
\author[N. Borie]{Nicolas Borie}

\title[Combinatorics of simple marked mesh patterns in $132$-avoiding permutations]{Combinatorics of simple marked mesh patterns in $132$-avoiding permutations}

\address{Univ. Paris Est
  Marne-La-Vall\'ee, Laboratoire d'Informatique Gaspard Monge, Cit\'e
  Descartes, B\^at Copernic -- 5, bd Descartes Champs sur Marne 77454
  Marne-la-Vall\'ee Cedex 2, France}

\keywords{permutation statistics, marked mesh pattern, distribution,
  Catalan numbers, Narayana numbers}

\begin{document}
\maketitle
\begin{abstract}
  We present some combinatorial interpretations for coefficients
  appearing in series partitioning the permutations avoiding $132$
  along marked mesh patterns. We identify for patterns in which
  only one parameter is non zero the combinatorial family in bijection
  with $132$-avoiding permutations and also preserving the statistic
  counted by the marked mesh pattern. 
 \end{abstract}

\section{Introduction}
\label{sec:intro}

Mesh patterns were introduced by Brändén and
Claesson~\cite{brandenClaesson} to provide explicit expansions for
certain permutation statistics as (possibly infinite) linear
combinations of (classical) permutation patterns. This notion was
further studied by Kitaev, Remmel and Tiefenbruck in some series of
papers refining conditions on permutations and patterns. The present
paper focuses on what the trio Kitaev, Remmel and Tiefenbruck call
\emph{simple marked pattern in $132$-avoiding permutations}~\cite{KRT}.

Let \begin{math}\sigma = \sigma_1 \dots \sigma_n \end{math} be a
permutation written in one-line notation. We now consider the graph
of \begin{math}\sigma\end{math}, \begin{math}G(\sigma)\end{math}, to
  be the set of points \begin{math}\{(i,\sigma_i): 1 \leqslant i
    \leqslant n \}\end{math}. For example, the graph of the
    permutation \begin{math}\sigma = 768945213 \end{math}
is represented Figure~\ref{mmp_draw}.

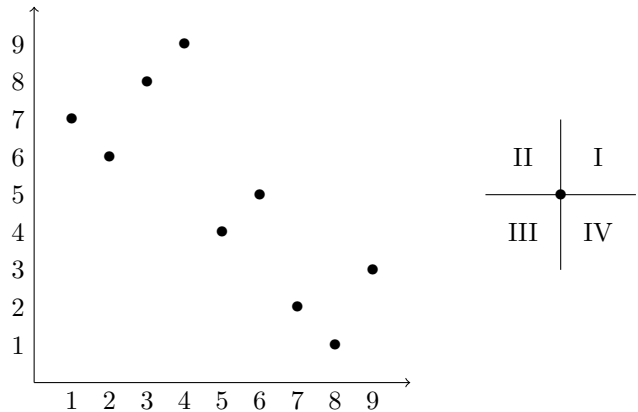
\begin{figure}
  \centering
  \begin{tikzpicture}[scale=0.5]
    \draw[->] (0,0) -- (0,10);
    \draw[->] (0,0) -- (10,0);

    \draw (1,7) node {$\bullet$};
    \draw (2,6) node {$\bullet$};
    \draw (3,8) node {$\bullet$};
    \draw (4,9) node {$\bullet$};
    \draw (5,4) node {$\bullet$};
    \draw (6,5) node {$\bullet$};
    \draw (7,2) node {$\bullet$};
    \draw (8,1) node {$\bullet$};
    \draw (9,3) node {$\bullet$};

    \foreach \y in {1,2,...,9} {
      \draw (0,\y) node[left] {$\y$};
      \draw (\y,0) node[below] {$\y$};
    }

    \draw (12,5) -- (16,5);
    \draw (14,3) -- (14,7);
    \draw (14,5) node {$\bullet$};
    \draw (15,6) node {I};
    \draw (13,6) node {II};
    \draw (13,4) node {III};
    \draw (15,4) node {IV};
  \end{tikzpicture}
  \caption{The graph of $\sigma = 768945213$ and the quadrant numeration.}
  \label{mmp_draw}
\end{figure}

Now, if we draw a coordinate system centered at the point $(i, \sigma_i)$,
we will be interested in the points that lie in the four quadrants I,
II, III, IV of that coordinate system as represented in
Figure~\ref{mmp_draw}. For any $a,b,c,d,$ four non negative integers,
we say that $\sigma_i$ matches the simple marked mesh pattern
$MMP(a,b,c,d)$ in $\sigma$ if, in the coordinate system centered at
$(i, \sigma_i)$, $G(\sigma)$ has at least $a$, $b$,
$c$ and $d$ points in the respective quadrants I, II, III and IV.

For $\sigma = 768945213$, then $\sigma_6 = 5$ matches
the simple marked mesh pattern $MMP(0,3,1,1)$, since relative to the
coordinate system with origin $(6,5)$, $G(\sigma)$ has respectively 0, 4, 1 and 3
points in Quadrants I, II, III and IV (see
Figure~\ref{ex_mmp_draw}). If a coordinate in $MMP(a,b,c,d)$ is zero,
then there is no condition imposed on the points in the corresponding
quadrant. We let $mmp(a,b,c,d)(\sigma)$ denote the number of $i$ such
that $\sigma_i$ matches the marked mesh pattern $MMP(a,b,c,d)$ in
$\sigma$. 

\begin{figure}
  \centering
  \begin{tikzpicture}[scale=0.5]
    \draw[->] (0,0) -- (0,10);
    \draw[->] (0,0) -- (10,0);

    \draw (1,7) node {$\bullet$};
    \draw (2,6) node {$\bullet$};
    \draw (3,8) node {$\bullet$};
    \draw (4,9) node {$\bullet$};
    \draw (5,4) node {$\bullet$};
    \draw (6,5) node {$\bullet$};
    \draw (7,2) node {$\bullet$};
    \draw (8,1) node {$\bullet$};
    \draw (9,3) node {$\bullet$};

    \foreach \y in {1,2,...,9} {
      \draw (0,\y) node[left] {$\y$};
      \draw (\y,0) node[below] {$\y$};
    }

    \draw (0.5,5) -- (9.5,5);
    \draw (6,0.5) -- (6,9.5);

    \draw (8,7) node {I};
    \draw (4,7) node {II};
    \draw (4,3) node {III};
    \draw (8,3) node {IV};

  \end{tikzpicture}
  \caption{The graph of $\sigma = 768945213$ with coordinate system at position $6$.}
  \label{ex_mmp_draw}
\end{figure}
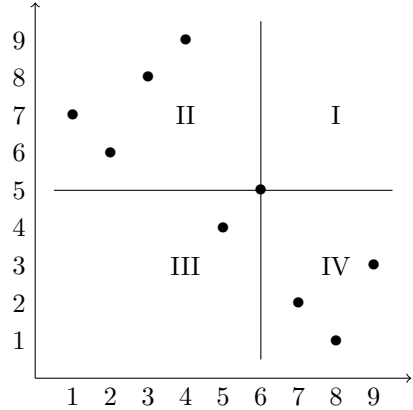

Given a sequence $w = w_1 \dots w_n$ of distinct integers, let
$pack(w)$ be the permutation obtained by replacing the $i$-th largest
integer that appears in $w$ by $i$. For example, if $w = 2754$, then
$pack(w) = 1432$. Given a permutation $\tau = \tau_1 \dots \tau_j$ in
the symmetric group $\SG_j$, we say that the pattern $\tau$ occurs in a
permutation $\sigma \in \SG_n$ if there exist $1
\leqslant i_1 \leqslant \dots \leqslant i_j \leqslant n$ such that
$pack(\sigma_{i_1} \dots \sigma_{i_j}) = \tau$. We say that a
permutation $\sigma$ avoids the pattern $\tau$ if $\tau$ does not
occur in $\sigma$. We will denote by $\SG_n(\tau)$ the set of permutations
in $\SG_n$ avoiding $\tau$. This paper presents some combinatorial
results concerning the generating functions
\begin{equation}
Q_{132}^{(a,b,c,d)}(t, x) := 1 + \sum_{n \geqslant 1} t^n Q_{n, 132}^{(a,b,c,d)}(x)
\end{equation}
where for any $a,b,c,d \in \NN$,
\begin{equation}
Q_{n, 132}^{(a,b,c,d)}(x) := \sum_{\sigma \in \SG_n(132)} x^{mmp(a,b,c,d)(\sigma)}.
\end{equation}
More precisely, we give a combinatorial interpretation of each
coefficient of these series when exactly one value among $a,b,c,d$ is
non zero. Not that, as explained in~\cite{KRT}, $Q_{132}^{(0,0,0,k)}(t,
x) = Q_{132}^{(0,k,0,0)}(t, x)$, so that we shall not consider the
forth quadrant.

\section{The quadrant I: the series $Q_{132}^{(\ell,0,0,0)}$}
\label{sec:k000}

The marked mesh pattern $MMP(\ell,0,0,0)$ splits the $132$-avoiding
permutations according to the number of values having at least $k$ greater
values on their right. The following table displays the first values of
$Q_{132}^{(\ell,0,0,0)}|_{t^nx^k}$:
\begin{equation}
\begin{array}{|c|ccccc|cccc|ccc|cc|c|}
\hline
              & \multicolumn{5}{c|}{\ell = {\bf 1}} & \multicolumn{4}{c|}{\ell = {\bf 2}} & \multicolumn{3}{c|}{\ell = {\bf 3}} & \multicolumn{2}{c|}{\ell = {\bf 4}} & \multicolumn{1}{c|}{\ell = {\bf 5}}\\ \hline
n \backslash k & {\bf 0} & {\bf 1} & {\bf 2} & {\bf 3} & {\bf 4} & {\bf 0} & {\bf 1} & {\bf 2} & {\bf 3} & {\bf 0} & {\bf 1} & {\bf 2} & {\bf 0} & {\bf 1} & {\bf 0} \\ \hline
{\bf 1}        & 1 &   &   &   &   & 1 &   &   &   & 1 &   &   & 1 &   & 1 \\ \hline 
{\bf 2}        & 1 & 1 &   &   &   & 2 &   &   &   & 2 &   &   & 2 &   & 2 \\ \hline 
{\bf 3}        & 1 & 2 & 2 &   &   & 4 & 1 &   &   & 5 &   &   & 5 &   & 5 \\ \hline 
{\bf 4}        & 1 & 3 & 5 & 5 &   & 8 & 4 & 2 &   &13 & 1 &   &14 &   &14 \\ \hline 
{\bf 5}        & 1 & 4 & 9 &14 &14 &16 &12 & 9 & 5 &34 & 6 & 2 &41 & 1 &42 \\ \hline 
\end{array}
\end{equation}

A lot of different combinatorial families are counted by the Catalan
numbers, and here Dyck paths appears to be the objects having a nice
behavior with the statistic associated with the pattern $MMP(\ell,0,0,0)$.

The Catalan triangle $(C_{n,k})_{1 \leqslant k \leqslant n}$ is (A009766 of~\cite{Sloane}):
\begin{equation}
\label{Tri-Catalan}
\begin{array}{c|ccccccccccccc}
n\backslash k& 1 & 2 & 3 & 4 & 5 & 6 & 7 & 8 \\
\hline
 1 & 1 &   &   &   &   &    &    &    \\
 2 & 1 & 1 &   &   &   &    &    &    \\
 3 & 1 & 2 & 2 &   &   &    &    &    \\
 4 & 1 & 3 & 5 & 5 &   &    &    &    \\
 5 & 1 & 4 & 9 & 14& 14&    &    &    \\
 6 & 1 & 5 & 14& 28& 42& 42 &    &    \\
 7 & 1 & 6 & 20& 48& 90&132 &132 &    \\
 8 & 1 & 7 & 27& 75&165&297 &429 & 429\\
\end{array}
\end{equation}
There is a nice formula for computing directly any coefficient
\begin{equation}
C_{n,k} = \frac{(n+k)!(n-k+1)}{k!(n+1)!}.
\end{equation}

We recall that a Dyck path with $2n$ steps is a path on the square lattice with steps
$(1, 1)$ or $(1, -1)$ from $(0, 0)$ to $(2n, 0)$ that never falls
below the x-axis.

\begin{theorem}
\label{thm-k000}
$Q_{132}^{(\ell,0,0,0)}|_{t^nx^k}$ counts the number of Dyck paths
  with $2n$ steps having exactly $k$ steps $(1,-1)$ ending at height
  greater than or equal to $\ell$.

\begin{proof}
This can be seen thanks to Krattenthaler's
bijection~\cite{Krat_byj} between $132$-avoiding permutations and Dyck
paths. The bijection consists in starting from the empty path and
constructing progressively a Dyck path by a left to right reading of
the permutation $\sigma = \sigma_1 \dots \sigma_n$. For each value
$\sigma_i$, we complete the path with some increasing
steps followed by a single decreasing step so that the height of the
path at this actual point is equal to the number of remaining values
$\sigma_j$ for $j>i$ greater than $\sigma_i$ (see
Figure~\ref{krat_dyck}).
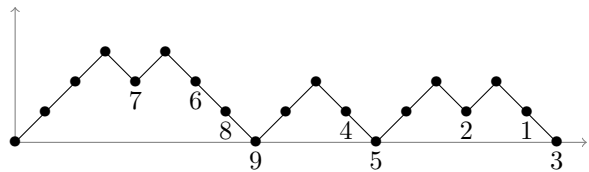
\begin{figure}[h]
  \centering
\begin{tikzpicture}[scale=0.4]
  \draw[->,gray] (0,0) -- (19,0);
  \draw[->,gray] (0,0) -- (0,4.5);

  \draw (0,0) node {$\bullet$};  
  \draw (0,0) -- (1,1);
  \draw (1,1) node {$\bullet$};
  \draw (1,1) -- (2,2);
  \draw (2,2) node {$\bullet$};
  \draw (2,2) -- (3,3);
  \draw (3,3) node {$\bullet$};
  \draw (3,3) -- (4,2);
  \draw (4,2) node {$\bullet$};
  \draw[below] (4,2) node {$7$};
  \draw (4,2) -- (5,3);
  \draw (5,3) node {$\bullet$};
  \draw (5,3) -- (6,2);
  \draw (6,2) node {$\bullet$};
  \draw[below] (6,2) node {$6$};
  \draw (6,2) -- (7,1);
  \draw (7,1) node {$\bullet$};
  \draw[below] (7,1) node {$8$};
  \draw (7,1) -- (8,0);
  \draw (8,0) node {$\bullet$};
  \draw[below] (8,0) node {$9$};
  \draw (8,0) -- (9,1);
  \draw (9,1) node {$\bullet$};
  \draw (9,1) -- (10,2);
  \draw (10,2) node {$\bullet$};
  \draw (10,2) -- (11,1);
  \draw (11,1) node {$\bullet$};
  \draw[below] (11,1) node {$4$};
  \draw (11,1) -- (12,0);
  \draw (12,0) node {$\bullet$};
  \draw[below] (12,0) node {$5$};
  \draw (12,0) -- (13,1);
  \draw (13,1) node {$\bullet$};
  \draw (13,1) -- (14,2);
  \draw (14,2) node {$\bullet$};
  \draw (14,2) -- (15,1);
  \draw (15,1) node {$\bullet$};
  \draw[below] (15,1) node {$2$};
  \draw (15,1) -- (16,2);
  \draw (16,2) node {$\bullet$};
  \draw (16,2) -- (17,1);
  \draw (17,1) node {$\bullet$};
  \draw[below] (17,1) node {$1$};
  \draw (17,1) -- (18,0);
  \draw (18,0) node {$\bullet$};
  \draw[below] (18,0) node {$3$};
\end{tikzpicture}
\caption{Dyck path obtained with the Krattenthaler's bijection over
  $\sigma = 768945213$.}~\label{krat_dyck}
\end{figure}
On the permutation $\sigma$ of Figure~\ref{mmp_draw}, $7$ sees two
greater values on its right, and so does $6$. $8$ sees only the $9$ which
itself cannot see any greater value as it is the highest. The proof is
now immediate since Krattenthaler prove that this construction
constitutes a bijection. By construction, $mmp(\ell,0,0,0)$ is equal to
the number of points at height greater or equal to $\ell$ reached by a
decreasing step. This gives another proof of Theorem 2 of~\cite{KRT}.
\end{proof}
\end{theorem}

As the number of increasing steps from height at least $1$ in Dyck paths
are counted by the Catalan triangle, we deduce the following
corollary.

\begin{corollary}
$Q_{132}^{(1,0,0,0)}|_{t^nx^k}$ is equal to the value $C_{n,k}$ of the
  Catalan triangle.
\end{corollary}

This result is a refinement of Theorem 5 of~\cite{KRT}.

To illustrate Theorem~\ref{thm-k000}, here is an example with $\ell =
2$ and $n = 4$. We have $Q_{4, 132}^{(2,0,0,0)}(x) = 8 + 4x + 2x^2$
and Figure~\ref{dyck_4} shows the fourteen Dyck paths of length $8$
with marked decreasing steps ending at height greater than or equal to
$2$. We have eight Dyck paths under the horizontal line at height $2$,
four paths containing a single decreasing step ending at height $2$,
and two last paths containing $2$ such decreasing steps.

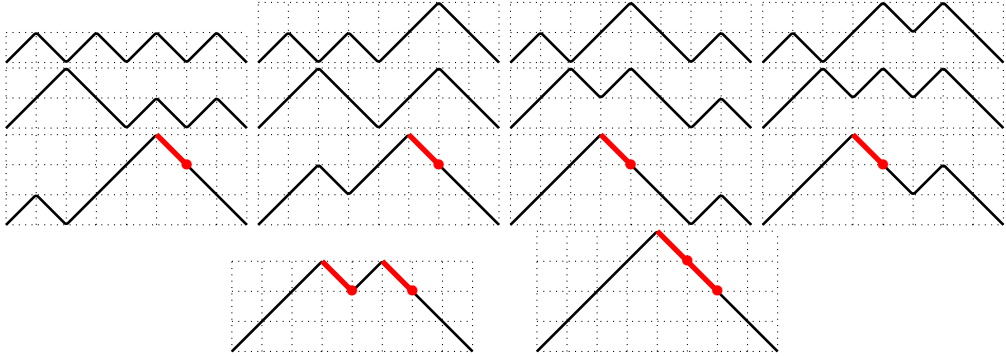
\begin{figure}
  \begin{tikzpicture}[scale=0.4]
  \draw[dotted] (0, 0) grid (8, 1);
  \draw[rounded corners=1, color=black, line width=1] (0, 0) -- (1, 1) -- (2, 0) -- (3, 1) -- (4, 0) -- (5, 1) -- (6, 0) -- (7, 1) -- (8, 0);
  \end{tikzpicture}
\begin{tikzpicture}[scale=0.4]
  \draw[dotted] (0, 0) grid (8, 2);
  \draw[rounded corners=1, color=black, line width=1] (0, 0) -- (1, 1) -- (2, 0) -- (3, 1) -- (4, 0) -- (5, 1) -- (6, 2) -- (7, 1) -- (8, 0);
\end{tikzpicture}
\begin{tikzpicture}[scale=0.4]
  \draw[dotted] (0, 0) grid (8, 2);
  \draw[rounded corners=1, color=black, line width=1] (0, 0) -- (1, 1) -- (2, 0) -- (3, 1) -- (4, 2) -- (5, 1) -- (6, 0) -- (7, 1) -- (8, 0);
\end{tikzpicture}
\begin{tikzpicture}[scale=0.4]
  \draw[dotted] (0, 0) grid (8, 2);
  \draw[rounded corners=1, color=black, line width=1] (0, 0) -- (1, 1) -- (2, 0) -- (3, 1) -- (4, 2) -- (5, 1) -- (6, 2) -- (7, 1) -- (8, 0);
\end{tikzpicture}
\begin{tikzpicture}[scale=0.4]
  \draw[dotted] (0, 0) grid (8, 2);
  \draw[rounded corners=1, color=black, line width=1] (0, 0) -- (1, 1) -- (2, 2) -- (3, 1) -- (4, 0) -- (5, 1) -- (6, 0) -- (7, 1) -- (8, 0);
\end{tikzpicture}
\begin{tikzpicture}[scale=0.4]
  \draw[dotted] (0, 0) grid (8, 2);
  \draw[rounded corners=1, color=black, line width=1] (0, 0) -- (1, 1) -- (2, 2) -- (3, 1) -- (4, 0) -- (5, 1) -- (6, 2) -- (7, 1) -- (8, 0);
\end{tikzpicture}
\begin{tikzpicture}[scale=0.4]
  \draw[dotted] (0, 0) grid (8, 2);
  \draw[rounded corners=1, color=black, line width=1] (0, 0) -- (1, 1) -- (2, 2) -- (3, 1) -- (4, 2) -- (5, 1) -- (6, 0) -- (7, 1) -- (8, 0);
\end{tikzpicture}
\begin{tikzpicture}[scale=0.4]
  \draw[dotted] (0, 0) grid (8, 2);
  \draw[rounded corners=1, color=black, line width=1] (0, 0) -- (1, 1) -- (2, 2) -- (3, 1) -- (4, 2) -- (5, 1) -- (6, 2) -- (7, 1) -- (8, 0);
\end{tikzpicture}
\begin{tikzpicture}[scale=0.4]
  \draw[dotted] (0, 0) grid (8, 3);
  \draw[rounded corners=1, color=black, line width=1] (0, 0) -- (1, 1) -- (2, 0) -- (3, 1) -- (4, 2) -- (5, 3) -- (6, 2) -- (7, 1) -- (8, 0);
  \draw[rounded corners=1, color=red, line width=2] (5, 3) -- (6, 2);
  \draw[color=red] (6,2) node {$\bullet$};
\end{tikzpicture}
\begin{tikzpicture}[scale=0.4]
  \draw[dotted] (0, 0) grid (8, 3);
  \draw[rounded corners=1, color=black, line width=1] (0, 0) -- (1, 1) -- (2, 2) -- (3, 1) -- (4, 2) -- (5, 3) -- (6, 2) -- (7, 1) -- (8, 0);
  \draw[rounded corners=1, color=red, line width=2] (5, 3) -- (6, 2);
  \draw[color=red] (6,2) node {$\bullet$};
\end{tikzpicture}
\begin{tikzpicture}[scale=0.4]
  \draw[dotted] (0, 0) grid (8, 3);
  \draw[rounded corners=1, color=black, line width=1] (0, 0) -- (1, 1) -- (2, 2) -- (3, 3) -- (4, 2) -- (5, 1) -- (6, 0) -- (7, 1) -- (8, 0);
  \draw[rounded corners=1, color=red, line width=2] (3, 3) -- (4, 2);
  \draw[color=red] (4,2) node {$\bullet$};
\end{tikzpicture}
\begin{tikzpicture}[scale=0.4]
  \draw[dotted] (0, 0) grid (8, 3);
  \draw[rounded corners=1, color=black, line width=1] (0, 0) -- (1, 1) -- (2, 2) -- (3, 3) -- (4, 2) -- (5, 1) -- (6, 2) -- (7, 1) -- (8, 0);
  \draw[rounded corners=1, color=red, line width=2] (3, 3) -- (4, 2);
  \draw[color=red] (4,2) node {$\bullet$};
\end{tikzpicture}
\begin{tikzpicture}[scale=0.4]
  \draw[dotted] (0, 0) grid (8, 3);
  \draw[rounded corners=1, color=black, line width=1] (0, 0) -- (1, 1) -- (2, 2) -- (3, 3) -- (4, 2) -- (5, 3) -- (6, 2) -- (7, 1) -- (8, 0);
  \draw[rounded corners=1, color=red, line width=2] (5, 3) -- (6, 2);
  \draw[rounded corners=1, color=red, line width=2] (3, 3) -- (4, 2);
  \draw[color=red] (4,2) node {$\bullet$};
  \draw[color=red] (6,2) node {$\bullet$};
\end{tikzpicture}
\qquad
\begin{tikzpicture}[scale=0.4]
  \draw[dotted] (0, 0) grid (8, 4);
  \draw[rounded corners=1, color=black, line width=1] (0, 0) -- (1, 1) -- (2, 2) -- (3, 3) -- (4, 4) -- (5, 3) -- (6, 2) -- (7, 1) -- (8, 0);
  \draw[rounded corners=1, color=red, line width=2] (4, 4) -- (5, 3) -- (6, 2);
  \draw[color=red] (5,3) node {$\bullet$};
  \draw[color=red] (6,2) node {$\bullet$};
\end{tikzpicture}
  \caption{Decreasing steps ending at height greater than or equal to $2$ in Dyck paths of length $8$.}~\label{dyck_4}
\end{figure}

\section{The quadrant II: the series $Q_{132}^{(0,\ell,0,0)}$}
\label{sec:0k00}

The quadrant II is now about to classify $132$-avoiding permutations
according to the number of values having $k$ greater values on their
left. We will see that this statistic can be read on first values
only. The following table displays the first values of
$Q_{132}^{(0,\ell,0,0)}|_{t^nx^k}$:

\begin{equation}
\begin{array}{|c|ccccc|cccc|ccc|cc|c|}
\hline
              & \multicolumn{5}{c|}{\ell = {\bf 1}} & \multicolumn{4}{c|}{\ell = {\bf 2}} & \multicolumn{3}{c|}{\ell = {\bf 3}} & \multicolumn{2}{c|}{\ell = {\bf 4}} & \multicolumn{1}{c|}{\ell = {\bf 5}}\\ \hline
n \backslash k & {\bf 0} & {\bf 1} & {\bf 2} & {\bf 3} & {\bf 4} & {\bf 0} & {\bf 1} & {\bf 2} & {\bf 3} & {\bf 0} & {\bf 1} & {\bf 2} & {\bf 0} & {\bf 1} & {\bf 0} \\ \hline
{\bf 1}        & 1 &   &   &   &   & 1 &   &   &   & 1 &   &   & 1 &   & 1 \\ \hline 
{\bf 2}        & 1 & 1 &   &   &   & 2 &   &   &   & 2 &   &   & 2 &   & 2 \\ \hline 
{\bf 3}        & 1 & 2 & 2 &   &   & 3 & 2 &   &   & 5 &   &   & 5 &   & 5 \\ \hline 
{\bf 4}        & 1 & 3 & 5 & 5 &   & 4 & 6 & 4 &   & 9 & 5 &   &14 &   &14 \\ \hline 
{\bf 5}        & 1 & 4 & 9 &14 &14 & 5 &12 &15 &10 &14 &18 &10 &28 &14 &42 \\ \hline 
\end{array}
\end{equation}

Let $n$ be a positive integer, we will call non-decreasing parking
functions of length $n$ the set of all non-decreasing functions $f$
from $1, 2, \dots, n$ into $1, 2, \dots, n$ such that for all $1
\leqslant i \leqslant n: f(i) \leqslant i$. It is well-known that
their are counted by the Catalan numbers. It is also known that they
are counted by the Catalan triangle~(see (\ref{Tri-Catalan})) if we refine by
the value of $f(n)$ and also counted by Narayana
numbers~(see (\ref{nara-Catalan})) if we refine by the number of different
values of $f$.

\subsection{The general result}

\begin{proposition}
\label{sigma1}
Let $\sigma$ be a $132$-avoiding permutation. The statistic
$mmp(0,1,0,0)$ only depends on the first value of $\sigma$. More
precisely, $mmp(0,1,0,0)(\sigma) = \sigma_1 - 1$.

\begin{proof}
Let $\sigma = \sigma_1 \dots \sigma_n$ be a $132$-avoiding
permutation. The $\sigma_1 - 1$ values smaller than $\sigma_1$ are to
the right of the first position thus see at least one greater value on
their left, hence $mmp(0,1,0,0)(\sigma) \geqslant \sigma_1 - 1$. The
$n - \sigma_1$ values greater than $\sigma_1$ must be ordered
increasingly otherwise such two values together with the first value
would form a $132$ pattern, so all these values cannot see something
greater left to themselves and $mmp(0,1,0,0)(\sigma) \leqslant
\sigma_1 - 1$.
\end{proof}
\end{proposition}

This first proposition is the key for the understanding of the
combinatorics attached to quadrant II. We will also understand it as
follows: if a position $i$ matches $MMP(0,1,0,0)$ in $\sigma$,
thus $\sigma_1 \geqslant \sigma_i$. The second idea consists in
applying this first proposition by induction. For that purpose, let us remove the
first value of $\sigma$ and \emph{pack} the remaining values. The
new first value now gives us information on values having two greater
values on their left.

\begin{theorem}
\label{thm-0k00}
$Q_{132}^{(0,\ell,0,0)}|_{t^nx^k}$ is equal to the number of
  $132$-avoiding permutations in $\SG_n$ such that the \emph{packed}
  of their suffix of length $n + 1 -\ell$ begins with value $k+1$.

\begin{proof}
Another interpretation of Proposition~\ref{sigma1} consists in noticing
that if a value $\sigma_i$ has a greater value on its left, the first
value $\sigma_1$ of the permutation is also greater than
$\sigma_i$. By recurrence, a value $\sigma_i$ having $\ell$ greater
values on their left is in fact smaller than the first $\ell$ values of
$\sigma$. This is due to the fact that if
$\sigma_{\ell} < \sigma_{\ell+1}$, then $\sigma_{\ell+1} =
\sigma_{\ell}+1$ ($132$ avoidance implies that the remaining greater
values must be ordered) otherwise $\sigma_{\ell} > \sigma_{\ell+1}$
and a value seeing $\sigma_{\ell+1}$ greater on its left sees also
$\sigma_{\ell}$ and by induction all previous values. Now, as having
$\ell$ greater values on the left is equivalent to be smaller than 
$\sigma_\ell$, this can be checked using the Proposition~\ref{sigma1}
on the suffix of length $n+1-\ell$ that we pack by commodity
allowing us to keep $132$-permutations (on smaller length when $\ell \geqslant 2$).
\end{proof}
\end{theorem}

For example, with $\ell=2$ and $n=4$, we have
$Q_{4,132}^{(0,2,0,0)}(x) = 4+6x+4x^2$. The following table displays
the fourteen permutations of length $14$ avoiding $132$, their suffix of
length $4+1-2 = 3$, the \emph{packed} of the suffix and the statistic
$mmp(0,2,0,0)$. The permutation are enumerated by lexicographic order,
however the reader can check that four of them contain no value having
two greater values on their left, six contain a single such value and
the last four contains two positions matching with $MMP(0,2,0,0)$.

\begin{equation}
  \begin{array}{|c|c|c|c|}
    \hline
    \sigma \text{ avoiding $132$} & \text{suffix of $\sigma$ of length $3$} & \text{packed of the suffix} & mmp(0,2,0,0)(\sigma) \\ \hline
    1234 & 234 & {\bf 1}23 & 0 \\
    2134 & 134 & {\bf 1}23 & 0 \\
    2314 & 314 & {\bf 2}13 & 1 \\
    2341 & 341 & {\bf 2}31 & 1 \\
    3124 & 124 & {\bf 1}23 & 0 \\
    3214 & 214 & {\bf 2}13 & 1 \\
    3241 & 241 & {\bf 2}31 & 1 \\
    3412 & 412 & {\bf 3}12 & 2 \\
    3421 & 421 & {\bf 3}21 & 2 \\
    4123 & 123 & {\bf 1}23 & 0 \\
    4213 & 213 & {\bf 2}13 & 1 \\
    4231 & 231 & {\bf 2}31 & 1 \\
    4312 & 312 & {\bf 3}12 & 2 \\
    4321 & 321 & {\bf 3}21 & 2 \\ \hline
  \end{array}
\end{equation}

\subsection{A new bijection between $132$-avoiding permutations and non-decreasing \\ parking functions}

Let $\sigma \in \SG_n$ be a $132$-avoiding permutation. We set
\begin{equation}
\phi(\sigma) := (mmp(0,n,0,0)+1, mmp(0,n-1,0,0)+1, \dots mmp(0,1,0,0)+1)
\end{equation}

\begin{theorem}
Let $n$ be a positive integer. $\phi$ establishes a bijection between
$132$-avoiding permutations of length $n$ and non-decreasing parking
functions of length $n$.

\begin{proof}
First, we check that $\phi$ is a map that builds a
non-decreasing parking function on length $n$. This is relatively easy
to see as a value having $k$ greater values on its left must be at
position at least $k+1$ so that $mmp(0,k,0,0) + 1 \leqslant k$. 
The fact that $\phi(\sigma)$ is non-decreasing comes from the fact
that values having $k$ greater values on their left thus have $k-1$
greater values on this same side.
We can build the inverse bijection using Proposition~\ref{sigma1}. We
start with the empty permutation, we insert at the left each value
from the non-decreasing parking function by a left to right reading
and we add $1$ to values greater than or equal to the new value
inserted (in this way, we keep a permutation at each step). Let us do
that on the parking function $(1,1,2,4,4,6,6,6,7)$.
\begin{equation}
  \begin{array}{rr}
    1 \text{ is read } & 1 \\
    1 \text{ is read } & 1(1+1) = 12 \\
    2 \text{ is read } & 21(2+1) = 213 \\
    4 \text{ is read } & 4213 \\
    4 \text{ is read } & 4(4+1)213 = 45213 \\
    6 \text{ is read } & 645213 \\
    6 \text{ is read } & 6(6+1)45213 = 6745213 \\
    6 \text{ is read } & 6(6+1)(7+1)45213 = 67845213 \\
    7 \text{ is read } & 76(7+1)(8+1)45213 = 768945213 \\
  \end{array}
\end{equation}
The reader can check that this non-decreasing parking
function corresponds to the permutation from
Figure~\ref{mmp_draw}. Proposition~\ref{sigma1} applied at each
step gives the relation between the $k^{th}$ value and
$mmp(0,k,0,0)(\sigma)$.
\end{proof}
\end{theorem}

Classified along their last part (which corresponds to $mmp(0,1,0,0)$
on the side of $132$-avoiding permutations), non-decreasing parking
function are counted by the Catalan triangle~(see
(\ref{Tri-Catalan})), we thus have the following corollary.

\begin{corollary}
The number of $132$-avoiding permutations containing $k$ values having
a greater value on their left is the coefficient $C_{n,k}$ of the
Catalan triangle: $Q_{132}^{(0,1,0,0)}|_{t^nx^k} = C_{n,k}$.
\end{corollary}

\subsection{The series $Q_{132}^{(0,k,0,0)}$ at $x=0$}

In the article~\cite{KRT}, the authors obtain by induction that
their function $Q_{132}^{(0,k,0,0)}(t,0)$ satisfies (Thm.~12 p.~24)
\begin{equation}
Q_{132}^{(0,k,0,0)}(t,0) = \frac{1+t \sum_{j=0}^{k-2} C_j t^j (Q_{132}^{(0,k-1-j,0,0)}(t,0) - 1) }{1-t}.
\end{equation}
The first examples are
\begin{equation}
\begin{split}
Q_{132}^{(0,1,0,0)}(t,0) &= \frac{1}{1-t}, \\
Q_{132}^{(0,2,0,0)}(t,0) &= \frac{1-t+t^2}{(1-t)^2}, \\
Q_{132}^{(0,3,0,0)}(t,0) &= \frac{1-2t+2t^2+t^3-t^4}{(1-t)^3}, \\
Q_{132}^{(0,4,0,0)}(t,0) &= \frac{1-3t+4t^2-t^3+3t^4-5t^5+2t^6}{(1-t)^4}.
\end{split}
\end{equation}
It happens that their series is easy to compute directly. In a further
paper, Kitaev and Liese explain that the coefficients are related to
the Catalan triangle~\cite{kitaev_catalan}.
\begin{theorem}
$Q_{n,132}^{(0,k,0,0)}(0)$ is equal to the sum of the $k$ first
  values of the $n^{th}$ row of the Catalan triangle.
\end{theorem}

From a combinatorics point of view, it is more simple to present
$Q_{n,132}^{(0,k,0,0)}(0)$ as series in $t$ in columns:
\begin{equation}
\label{LesQ}
\begin{array}{c|ccccccccccccc}
n\backslash k
   & 1 & 2 & 3 & 4 & 5 & 6 & 7 & 8 \\
\hline
 1 & 1 & 1 & 1 & 1 & 1 & 1 & 1  &   1 \\
 2 & 1 & 2 & 2 & 2 & 2 & 2 & 2  &   2 \\
 3 & 1 & 3 & 5 & 5 & 5 & 5 & 5  &   5 \\
 4 & 1 & 4 & 9 & 14& 14& 14& 14 &  14 \\
 5 & 1 & 5 & 14& 28& 42& 42& 42 &  42 \\
 6 & 1 & 6 & 20& 48& 90&132& 132& 132 \\
 7 & 1 & 7 & 27& 75&165&297& 429& 429 \\
 8 & 1 & 8 & 35&110&275&572&1001&1430 \\
\end{array}
\end{equation}

We thus recognize a simple recurrence
\begin{theorem}
\label{thm-0k00}
We have
\begin{equation}
Q_{n,132}^{(0,k,0,0)}(0) =
\left\{
\begin{array}{lr}
1 & \text{if\ } n = 1 \text{ or } k = 1, \\
Q_{n-1,132}^{(0,k,0,0)}(0) + Q_{n,132}^{(0,k-1,0,0)}(0)
& \text{if\ } n\geq k, \\
Q_{n,132}^{(0,k-1,0,0)}(0)
& \text{if\ } n<k. \\
\end{array}
\right.
\end{equation}
\end{theorem}

We can prove this property with at least three different approaches:
the first one is analytic by noticing that Formula (24)
presented by~\cite{KRT} at $x=0$ imply our relations, the second
method builds a bijection between the related sets, the last one shows
that simple combinatorics objets counted by these numbers satisfy the
induction.

\subsubsection{The analytic proof}

In order to simplify the notations, we define $R_n^k =
Q_{n,132}^{(0,k,0,0)}(t,0)$. Here $C_n$ denote the $n^{th}$ Catalan
number.

Formula (24) of~\cite{KRT} in $x=0$ becomes
\begin{equation}
R_n^k = R_{n-1}^k + \sum_{i=1}^{k-1} C_{i-1} R_{n-i}^{k-i}.
\end{equation}

If $n<k$, thus $n-1<k-1$ and $n-i<k-i$, hence by induction on $n+k$ and
using $R_n^n=C_n$, we get $R_n^k=R_n^{k-1}$.

If $n\geq k$, we suppose the theorem true for all pairs $(n',k')$
with $n'+k'<n+k$. We have
\begin{equation}
\begin{split}
R_n^k
 &= R_{n-1}^k + \sum_{i=1}^{k-2} C_{i-1} (R_{n-i}^{k-i-1}+R_{n-i-1}^{k-i})
              + C_{k-2} R_{n-k+1}^1 \\
 &= R_{n-1}^{k-1} + \sum_{i=1}^{k-2} C_{i-1} R_{n-i}^{k-i-1}
  + R_{n-2}^k     + \sum_{i=1}^{k-1} C_{i-1} R_{n-i-1}^{k-i})  \\
 &= R_{n-1}^k + R_{n}^{k-1}.
\end{split}
\end{equation}

\subsubsection{The bijective proof on permutations}

Let us denote $S(n,k)$ the set of permutations avoiding $132$ and avoiding the mesh pattern $(0,k,0,0)$. Then,

\begin{proposition}
\label{prop-snk}
The set $S(n,k)$ is composed of permutations of size $n$ avoiding
$132$ whose position $i$ of $1$ satisfy $i\leq k$. In particular,
$S(n,k) / S(n,k-1)$ consists in elements where $1$ is at the $k^{th}$
place.
It is equivalent to say that the descents composition of these
permutations have their last part equal to $n+1-k$.
\end{proposition}

\begin{proof}
Let $p$ be a permutation avoiding $132$ and let $i$ the position of its
$1$. If $i>k$, then $p$ does not avoid $(0,k,0,0)$. Otherwise, as $p$
avoids $132$, all values after $1$ must be in increasing order. That
implies that the number of points in quadrant $II$ is maximal for the
position $i$ and that it is bounded by $i-1 < k$.
From this remark concerning the order of the elements after $1$, we deduce the
interpretation in terms of descents composition of $p$.
\end{proof}

The set $S(n,k)$ splits into two blocks: the permutations for which
$1$ is at position $k$ and the other. The second ones are directly
element of $S(n,k-1)$. Permutation avoiding $132$ in which $1$ is in
position $k$ are in bijection with element of $S(n-1,k)$ using the
following process: remove the $1$ and pack the obtained
result.

For example, one can check that $S(7,3)$ is composed by the $27$
following permutations:
\begin{equation}
\begin{split}
1234567, 2134567, 2314567, 3124567, 3214567, 3412567, 4123567, 4213567,
 4312567, \\
 4512367, 5123467, 5213467, 5312467, 5412367, 5612347, 6123457, 6213457, 
 6312457, \\
 6412357, 6512347, 6712345, 7123456, 7213456, 7312456,
 7412356, 7512346, 7612345
\end{split}
\end{equation}

These $27$ permutations are obtained as the reunion of the $7$ elements
of $S(7,2)$
\begin{equation}
1234567, 2134567, 3124567, 4123567, 5123467, 6123457, 7123456
\end{equation}
and the $20$ elements of $S(6,3)$
\begin{equation}
\begin{split}
123456, 213456, 231456, 312456, 321456, 341256, 412356, 421356, 431256,
 451236, \\
 512346, 521346, 531246, 541236, 561234, 612345, 621345, 631245,
 641235, 651234
\end{split}
\end{equation}
in which we have incremented the values by one then added the $1$ in
third position.


\subsubsection{Proof using other combinatorics objects}

The Catalan triangle counts for example non decreasing parking
functions of size $n$ and maximum $k$. Let us denote $ND(n,k)$ the set
of non decreasing parking functions of size $n$ and maximum at most
$k$.
Then $Q_{n,132}^{(0,k,0,0)}(t,0)$ counts the cardinality of $ND(n,k)$.

We thus immediate check the induction formula of
Theorem~\ref{thm-0k00}:

either $p\in ND(n,k)$ ends with a value smaller than $k$ and is an
element of $ND(n,k-1)$, either $p$ ends by $k$ and in this case $p$
without its last part is an element of $ND(n-1,k)$. We point out that
if $n<k$, $p$ can end by $k$, that justifies how the second case is
empty and thus the two cases of the theorem.

Moreover, they also satisfy Formula (24) of~\cite{KRT} at $x=0$:
\begin{equation}
Q_{n,132}^{(0,k,0,0)}(t,0) =
\sum_{i=1}^{k-1} C_{i-1} Q_{n-i,132}^{(0,k-i,0,0)}(t,0) +
Q_{n-1,132}^{(0,k,0,0)}(t,0).
\end{equation}
This decomposition comes from the decomposition of non-decreasing
parking functions along their breakpoints: recall that the
breakpoints of $p$ are equal to the values $>1$ such that $p_i=i$. Si
$p$ has a breakpoint at position $i$, we deduce that $p=u \cdot v[|u|]$ with $u$ of
length $i-1$ and $v$ two non decreasing parking functions.

Formula (24) of~\cite{KRT} should be interpreted as follows: we
decompose each element of $ND(n,k)$ along their first breakpoint (if
they have one). Elements without breakpoint are prime parking
functions and are in bijection with elements of $ND(n-1,k)$ (we just
remove the first letter).  Otherwise, $p=u.v[|u|]$ where $v$ is
prime. Then if $|u|=i$, we write $u=1.u'$ where $u'$ is a parking
function of length $i-1$ and $v$ is a parking function of length $n-i$
of maximum at most $k-i$. It is obvious that the shifted concatenation
of any prime parking function of length $i$ with an element of
$ND(n-i,k-i)$ is also an element of $ND(n,k)$ and such obtained sets
when $i$ takes different values are all disjoint. That proves the formula.

\subsubsection{Connections with descents}


Proposition~\ref{prop-snk} shows that elements of $S(n,k)$ are the
permutations avoiding $132$ whose descent compositions is less fine than
$I(k)=(1^{k-1},n-k+1)$. Let us consider the left weak order $<$, the
permutohedron on values. Let us denote by $t(I)$ (resp. $s(I)$) the greatest (resp. smallest)
permutation avoiding $132$ for $<$ of descent composition $I$.

Let us recall that $<$ restricted on $132$ avoiding
permutations is the order called Tamari order and that we can also
obtain it as quotient of the weak order by grouping permutations along
Sylvester classes on theirs inverses~\cite{PBT}.

As descent classes are union of Sylvester classes on their inverses and
as they inherit a lattice structure for the refinement order, we
deduce
\begin{proposition}
For all $n$ and all $k$, elements of $S(n,k)$ are the interval
$[id,t(1^{k-1},n-k+1)]$ for the Tamari order.

Moreover, the set $S(n,k) / S(n,k-1)$ is the interval
$[s(k-1,n-k+1),t(1^{k-1},n-k+1)]$. In particular, $S(n,k)$ splits into a union
of intervals for the Tamari order.
\end{proposition}

\begin{proof}
The first part comes from the previous discussion. The second part is
a consequence of the second part of Proposition~\ref{prop-snk} and
the fact that compositions ending by $k$ are all in the interval
$[(k-1,n-k+1),(1^{k-1},n-k+1)]$ in the order of descent compositions.
\end{proof}

For example, $S(6,4)$ is equal to the interval $[123456,654123]$ which
can be decomposed into four intervals
\begin{equation}
[123456,123456],\ 
[213456,612345],\
[231456,651234],\
[234156,654123].
\end{equation}

We get a generalization of this result for the other coefficients
of polynomials $Q_{n,{132}}$:

\begin{proposition}
Let $n$ be a positive integer. Let us consider the polynomial $Q_{n,132}^{(0,k,0,0)}(t,x)$.  
The permutations whose monomial in $x$
is smaller than a certain value $\ell$ is a superior ideal of the
Tamari order. We denote by $I(n,k,\ell)$ this ideal.
\end{proposition}

\begin{proof}
This proposition is immediate since if a permutation $\sigma$ does not
admit the mesh pattern $(0,k,0,0)$ in position $i$, no permutation of
the interval $[id,\sigma]$ could have this pattern.
\end{proof}

\section{The quadrant III: the series $Q_{132}^{(0,0,\ell,0)}$}
\label{sec:0010}

These series refine the $132$-avoiding permutations along the number
of values having $k$ smaller values on their left. The following table
displays the first values of $Q_{132}^{(0,0,\ell,0)}|_{t^nx^k}$:

\begin{equation}
\begin{array}{|c|ccccc|cccc|ccc|cc|c|}
\hline
              & \multicolumn{5}{c|}{\ell = {\bf 1}} & \multicolumn{4}{c|}{\ell = {\bf 2}} & \multicolumn{3}{c|}{\ell = {\bf 3}} & \multicolumn{2}{c|}{\ell = {\bf 4}} & \multicolumn{1}{c|}{\ell = {\bf 5}}\\ \hline
n \backslash k & {\bf 0} & {\bf 1} & {\bf 2} & {\bf 3} & {\bf 4} & {\bf 0} & {\bf 1} & {\bf 2} & {\bf 3} & {\bf 0} & {\bf 1} & {\bf 2} & {\bf 0} & {\bf 1} & {\bf 0} \\ \hline
{\bf 1}        & 1 &   &   &   &   & 1 &   &   &   & 1 &   &   & 1 &   & 1 \\ \hline 
{\bf 2}        & 1 & 1 &   &   &   & 2 &   &   &   & 2 &   &   & 2 &   & 2 \\ \hline 
{\bf 3}        & 1 & 3 & 1 &   &   & 3 & 2 &   &   & 5 &   &   & 5 &   & 5 \\ \hline 
{\bf 4}        & 1 & 6 & 6 & 1 &   & 5 & 7 & 2 &   & 9 & 5 &   &14 &   &14 \\ \hline 
{\bf 5}        & 1 &10 &20 &10 & 1 & 8 &21 &11 & 2 &18 &19 & 5 &28 &14 &42 \\ \hline 
\end{array}
\end{equation}

In their paper~\cite{KRT}, the three authors compute several terms of the
function $Q_{132}^{(0,0,1,0)}$. It begins as follows:
\begin{equation}
\begin{array}{l}
Q_{132}^{(0,0,1,0)}(t, x) = 1 + t + (1+x)t^2 + (1+3x+x^2)t^3 + (1+6x+6x^2+x^3)t^4 \\
+ (1+10x+20x^2+10x^3+x^4)t^5 + (1+15x+50x^2+50x^3+15x^4+x^5)t^6 \\
+ (1+21x+105x^2+175x^3+105x^4+21x^5+x^6)t^7 \\
+ (1+28x+196x^2+490x^3+490x^4+196x^5+28x^6+x^7)t^8 + ...
\end{array}
\end{equation}

and we recall that the Narayana triangle (A001263 of~\cite{Sloane})
\begin{equation}
\label{nara-Catalan}
\begin{array}{c|ccccccccccccc}
n\backslash k& 1 & 2 & 3 & 4 & 5 & 6 & 7 & 8 \\
\hline
 1 & 1 &   &   &   &   &   &   &   \\
 2 & 1 & 1 &   &   &   &   &   &   \\
 3 & 1 & 3 & 1 &   &   &   &   &   \\
 4 & 1 & 6 & 6 & 1 &   &   &   &   \\
 5 & 1 &10 &20 &10 & 1 &   &   &   \\
 6 & 1 &15 &50 &50 &15 & 1 &   &   \\
 7 & 1 &21 &105&175&105&21 & 1 &   \\
 8 & 1 &28 &196&490&490&196&28 & 1 \\
\end{array}
\end{equation}
We will denote $N(n, k)$ the Narayana number indexed by $n$ and $k$.

\begin{theorem}
\label{thm-0010}
We have $\forall~n \geqslant 1, \forall~0 \leqslant k \leqslant n-1$,
\begin{equation}
Q_{132}^{(0,0,1,0)}(t, x)|_{t^nx^k} = N(n,k+1) = \frac{1}{n}\binom{n}{k}\binom{n}{k+1}
\end{equation}
\end{theorem}
This theorem precise Theorem 11 of~\cite{KRT}.

A way to see that is to consider the binary decreasing tree associated
to a $132$-avoiding permutation. Let us see that on an example, the
same permutation as in Figure~\ref{mmp_draw}: $\sigma =
768945213$. We build a binary tree by inserting the greatest value
(here $9$), then split the permutation into two parts, $768$ on the
left and $45213$ on the right. Values in the first part have to be
inserted on the left of this $9$ and the other part on the right. We
still begin by the greatest remaining values of each
sub-part. Therefore, the resulting tree is decreasing as we go deeper
in it (see Figure~\ref{bst_constr}).

\begin{figure}
  \centering
  \begin{tikzpicture}[scale=0.6]
    \draw[above] (3,5) node {9};
    \draw (3,5) node {$\bullet$};

    \draw (3,5) -- (2,4);
    \draw[above] (2,4) node {8};
    \draw (2,4) node {$\bullet$};

    \draw (2,4) -- (0,3);
    \draw[above] (0,3) node {7};
    \draw (0,3) node {$\bullet$};    

    \draw (0,3) -- (1,2);
    \draw[above] (1,2) node {6};
    \draw (1,2) node {$\bullet$};

    \draw (3,5) -- (5,4);
    \draw[above] (5,4) node {5};
    \draw (5,4) node {$\bullet$};

    \draw (5,4) -- (4,3);
    \draw[above] (4,3) node {4};
    \draw (4,3) node {$\bullet$};

    \draw (5,4) -- (8,3);
    \draw[above] (8,3) node {3};
    \draw (8,3) node {$\bullet$};

    \draw (8,3) -- (6,2);
    \draw[above] (6,2) node {2};
    \draw (6,2) node {$\bullet$};

    \draw (6,2) -- (7,1);
    \draw[above] (7,1) node {1};
    \draw (7,1) node {$\bullet$};

    \draw (-1,0.5) -- (9,0.5);
    \foreach \y in {0,1,...,8} {
      \draw (\y,0.6) -- (\y,0.4);
    }
    \draw[below] (0,0.5) node {7};
    \draw[below] (1,0.5) node {6};
    \draw[below] (2,0.5) node {8};
    \draw[below] (3,0.5) node {9};
    \draw[below] (4,0.5) node {4};
    \draw[below] (5,0.5) node {5};
    \draw[below] (6,0.5) node {2};
    \draw[below] (7,0.5) node {1};
    \draw[below] (8,0.5) node {3};

  \end{tikzpicture}
  \caption{Construction of the binary decreasing tree associated with $\sigma = 768945213$.}~\label{bst_constr}
\end{figure}
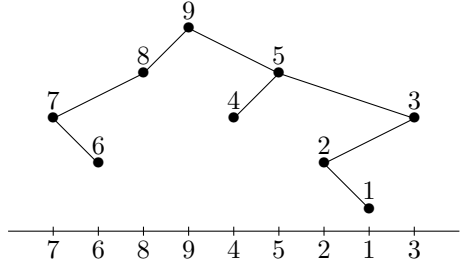

This construction is well-known to describe a bijection as we can
recover the permutation by gravity (or an infix deep first search
reading of the tree). Moreover, if we forget the labels and only keep
the binary tree shape, the previous map is still a bijection between
$132$-avoiding permutations and planar binary trees. The avoidance of
pattern $132$ makes that the tree can be uniquely labeled by $1, 2,
\dots , n$ decreasingly. The labels must be set from $n$, then $n-1,
\dots , 1$ following a prefix deep first search of the tree. As prefix
deep first search consists in parsing the root, the left child then
the right child all recursively, the resulting tree is therefore
decreasing and does not contain any $132$ pattern.

In the permutation $\sigma = 768945213$, the four values $8$, $9$, $5$
and $3$ see at least a point in quadrant III. These values correspond
exactly to the nodes of the corresponding tree having a left child.

\begin{proposition}
\label{bst_mmp}
Let $\sigma$ be a $132$-avoiding permutation of $\SG$, the number
$mmp(0,0,1,0)$ of positions matching with $MMP(0,0,1,0)$ is equal to
the number of left branches in the binary decreasing tree associated to
$\sigma$.

\begin{proof}
This fact is obvious due to the construction. A value to the left and
smaller than a given value must be inserted on the left and deeper as
the tree is decreasing, and on the other side, a node having a left
child sees at least one smaller value on its left. As the tree is
decreasing and values inserted on the left are left in the
permutation, any node having a left child matches the pattern
$MMP(0,0,1,0)$.
\end{proof}
\end{proposition}

\begin{proposition}
\label{nara_bst}
The number of binary trees over $n$ nodes containing $k$ left (or right)
branches is counted by the Narayana number $N(n,k+1)$.

\end{proposition}

Proposition~\ref{nara_bst} and the previous observation conclude
Theorem~\ref{thm-0010}. We know extend this previous statement from
pattern $(0,0,1,0)$ to pattern $(0,0,\ell,0)$ with $\ell$ any positive
integer.

\begin{theorem}
\label{thm-00k0}
$Q_{132}^{(0,0,\ell,0)}|_{t^nx^k}$ is equal to the number of binary
trees over $n$ nodes containing themselves $k$ left sub-trees over
at least $\ell$ nodes.

\begin{proof}
This is a straight-forward generalization. The bijection between binary
trees and $132$-avoiding permutations contains this result,
a position matching with $MMP(0,0,\ell,0)$ has a value that sees at least
$\ell$ smaller values on its left. All these smaller and left values are
inserted to the left of the concerning node in the binary decreasing
tree. This give another proof of the second part of Theorem 3 of~\cite{KRT}.
\end{proof}
\end{theorem}

For example, with $\ell=2$ and $n=4$, we have $Q_{4,
  132}^{(0,0,\ell,0)}(x) = 5 + 7x + 2x^2$. As one can check on Figure~\ref{dec_trees_4}
displaying the fourteen binary trees over $4$ nodes in which the left sub-trees
over at least $2$ nodes have been circled.

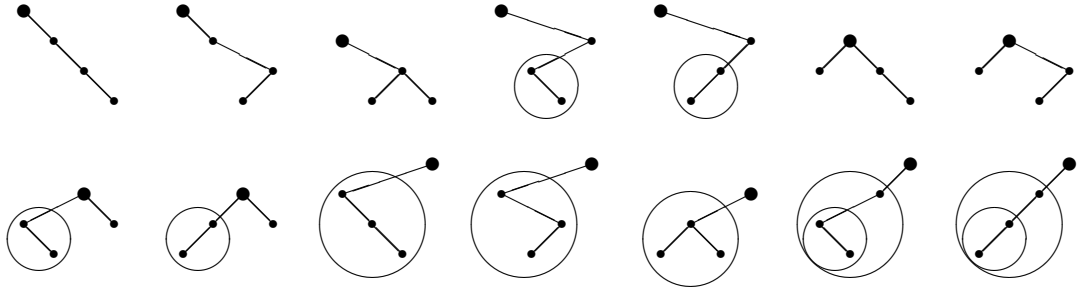
\begin{figure}
{\setlength\unitlength{4mm}

\begin{picture}(5,5)
\put(1,4){\circle*{0.3}}
\put(2,3){\circle*{0.3}}
\put(3,2){\circle*{0.3}}
\put(4,1){\circle*{0.3}}
\put(3,2){\line(1,-1){1}}
\put(2,3){\line(1,-1){1}}
\put(1,4){\line(1,-1){1}}
\put(1,4){\circle*{0.4}}
\end{picture}
\begin{picture}(5,5)
\put(1,4){\circle*{0.3}}
\put(2,3){\circle*{0.3}}
\put(3,1){\circle*{0.3}}
\put(4,2){\circle*{0.3}}
\put(4,2){\line(-1,-1){1}}
\put(2,3){\line(2,-1){2}}
\put(1,4){\line(1,-1){1}}
\put(1,4){\circle*{0.4}}
\end{picture}
\begin{picture}(5,4)
\put(1,3){\circle*{0.3}}
\put(2,1){\circle*{0.3}}
\put(3,2){\circle*{0.3}}
\put(4,1){\circle*{0.3}}
\put(3,2){\line(-1,-1){1}}
\put(3,2){\line(1,-1){1}}
\put(1,3){\line(2,-1){2}}
\put(1,3){\circle*{0.4}}
\end{picture}
\begin{picture}(5,5) 
\put(1,4){\circle*{0.3}}
\put(2,2){\circle*{0.3}}
\put(3,1){\circle*{0.3}}
\put(2,2){\line(1,-1){1}}
\put(4,3){\circle*{0.3}}
\put(4,3){\line(-2,-1){2}}
\put(1,4){\line(3,-1){3}}
\put(1,4){\circle*{0.4}}
\put(2.5,1.5){\circle{2}}
\end{picture}
\begin{picture}(5,5) 
\put(1,4){\circle*{0.3}}
\put(2,1){\circle*{0.3}}
\put(3,2){\circle*{0.3}}
\put(3,2){\line(-1,-1){1}}
\put(4,3){\circle*{0.3}}
\put(4,3){\line(-1,-1){1}}
\put(1,4){\line(3,-1){3}}
\put(1,4){\circle*{0.4}}
\put(2.5,1.5){\circle{2}}
\end{picture}
\begin{picture}(5,4)
\put(1,2){\circle*{0.3}}
\put(2,3){\circle*{0.3}}
\put(3,2){\circle*{0.3}}
\put(4,1){\circle*{0.3}}
\put(3,2){\line(1,-1){1}}
\put(2,3){\line(-1,-1){1}}
\put(2,3){\line(1,-1){1}}
\put(2,3){\circle*{0.4}}
\end{picture}
\begin{picture}(5,4)
\put(1,2){\circle*{0.3}}
\put(2,3){\circle*{0.3}}
\put(3,1){\circle*{0.3}}
\put(4,2){\circle*{0.3}}
\put(4,2){\line(-1,-1){1}}
\put(2,3){\line(-1,-1){1}}
\put(2,3){\line(2,-1){2}}
\put(2,3){\circle*{0.4}}
\end{picture}
\begin{picture}(5,4)  
\put(1,2){\circle*{0.3}}
\put(2,1){\circle*{0.3}}
\put(1,2){\line(1,-1){1}}
\put(3,3){\circle*{0.3}}
\put(4,2){\circle*{0.3}}
\put(3,3){\line(-2,-1){2}}
\put(3,3){\line(1,-1){1}}
\put(3,3){\circle*{0.4}}
\put(1.5,1.5){\circle{2}}
\end{picture}
\begin{picture}(5,4) 
\put(1,1){\circle*{0.3}}
\put(2,2){\circle*{0.3}}
\put(2,2){\line(-1,-1){1}}
\put(3,3){\circle*{0.3}}
\put(4,2){\circle*{0.3}}
\put(3,3){\line(-1,-1){1}}
\put(3,3){\line(1,-1){1}}
\put(3,3){\circle*{0.4}}
\put(1.5,1.5){\circle{2}}
\end{picture}
\begin{picture}(5,5) 
\put(1,3){\circle*{0.3}}
\put(2,2){\circle*{0.3}}
\put(3,1){\circle*{0.3}}
\put(2,2){\line(1,-1){1}}
\put(1,3){\line(1,-1){1}}
\put(4,4){\circle*{0.3}}
\put(4,4){\line(-3,-1){3}}
\put(4,4){\circle*{0.4}}
\put(2,2){\circle{3.5}}
\end{picture}
\begin{picture}(5,5) 
\put(1,3){\circle*{0.3}}
\put(2,1){\circle*{0.3}}
\put(3,2){\circle*{0.3}}
\put(3,2){\line(-1,-1){1}}
\put(1,3){\line(2,-1){2}}
\put(4,4){\circle*{0.3}}
\put(4,4){\line(-3,-1){3}}
\put(4,4){\circle*{0.4}}
\put(1.75,2){\circle{3.5}}
\end{picture}
\begin{picture}(5,4) 
\put(1,1){\circle*{0.3}}
\put(2,2){\circle*{0.3}}
\put(3,1){\circle*{0.3}}
\put(2,2){\line(-1,-1){1}}
\put(2,2){\line(1,-1){1}}
\put(4,3){\circle*{0.3}}
\put(4,3){\line(-2,-1){2}}
\put(4,3){\circle*{0.4}}
\put(2,1.5){\circle{3}}
\end{picture}
\begin{picture}(5,5) 
\put(1,2){\circle*{0.3}}
\put(2,1){\circle*{0.3}}
\put(1,2){\line(1,-1){1}}
\put(3,3){\circle*{0.3}}
\put(3,3){\line(-2,-1){2}}
\put(4,4){\circle*{0.3}}
\put(4,4){\line(-1,-1){1}}
\put(4,4){\circle*{0.4}}
\put(2,2){\circle{3.5}}
\put(1.5,1.5){\circle{2}}
\end{picture}
\begin{picture}(5,5) 
\put(1,1){\circle*{0.3}}
\put(2,2){\circle*{0.3}}
\put(2,2){\line(-1,-1){1}}
\put(3,3){\circle*{0.3}}
\put(3,3){\line(-1,-1){1}}
\put(4,4){\circle*{0.3}}
\put(4,4){\line(-1,-1){1}}
\put(4,4){\circle*{0.4}}
\put(2,2){\circle{3.5}}
\put(1.5,1.5){\circle{2}}
\end{picture}}
\caption{Decreasing trees of $132$-avoiding permutations of length $4$.}~\label{dec_trees_4}
\end{figure}

\section*{acknowledgements}
\label{sec:ack}

The author would like to thanks Jean-Christophe Novelli for useful
discussions and comments. This research was driven by computer
exploration using the open-source mathematical software
\texttt{Sage}~\cite{Sage} and its algebraic combinatorics features
developed by the \texttt{Sage-Combinat}
community~\cite{Sage-Combinat}.

\bibliographystyle{alpha}
\bibliography{main}
\label{sec:biblio}

\end{document}